\documentclass{amsart}
\usepackage[utf8]{inputenc}
\usepackage[english]{babel}
\usepackage{amsmath}
\usepackage{amsthm}
\usepackage{amssymb}
\usepackage{mathtools}
\usepackage{tikz-cd}
\usepackage[T1]{fontenc}
\usepackage{wasysym}
\usepackage{dsfont}
\usepackage[many]{tcolorbox}
\usepackage{soul}
\usepackage{xcolor}
\usepackage{appendix}
\usepackage{enumerate}
\usepackage{multicol}
\usepackage{fancyhdr}
\usepackage{parskip}
\usepackage{hyperref}
\usepackage{mathtools}
\usepackage{nameref}
\usepackage{thmtools}
\usepackage{thm-restate}
\usepackage{adjustbox}
\usepackage{cleveref}
\usepackage{tikz}
\usepackage{float}
\usepackage{MnSymbol}
\usepackage{amsthm}
\usepackage{orcidlink}
\usepackage{multicol}
\usepackage{enumitem}
\usepackage{csquotes}
\usepackage[style=numeric]{biblatex}
\usetikzlibrary{graphs.standard} 

\hypersetup{
  colorlinks   = true, 
  urlcolor     = {blue!50!black}, 
  linkcolor    = {blue!50!black}, 
  citecolor   = {red!50!black} 
}
\newcommand{\obj}{\mathrm{obj}}
\newcommand{\morph}{\mathrm{morph}}
\newcommand{\Gra}{\mathbf{Gra}}

\newcommand{\End}{\mathrm{End}}

\newcommand{\Rel}{\mathbf{Rel}}

\renewcommand{\End}{\mathrm{End}}

\newtheorem{theorem}{Theorem}[section]
\newtheorem*{theorem*}{Theorem}
\newtheorem{lemma}[theorem]{Lemma}

\newtheorem*{proposition*}{Proposition}
\newtheorem{corollary}[theorem]{Corollary}
\newtheorem*{corollary*}{Corollary}
\newtheorem*{fact*}{Fact}

\newtheorem{fact}[theorem]{Fact}

\theoremstyle{definition}
\newtheorem*{definition*}{Definition}
\newtheorem{definition}[theorem]{Definition}

\newtheorem*{question*}{Question}

\newtheorem*{notation*}{Notation}

\theoremstyle{remark}

\newcommand{\Aut}{\mathsf{Aut}}

\newcommand{\Ind}[1]
{#1\setbox0=\hbox{$#1x$}\kern\wd0\hbox to 0pt{\hss$#1\mid$\hss} \lower.9\ht0\hbox to 0pt{\hss$#1\smile$\hss}\kern\wd0}

\newcommand{\notind}[1]
{#1\setbox0=\hbox{$#1x$}\kern\wd0
\hbox to 0pt{\mathchardef\nn=12854\hss$#1\nn$\kern1.4\wd0\hss}
\hbox to 0pt{\hss$#1\mid$\hss}\lower.9\ht0 \hbox to 0pt{\hss$#1\smile$\hss}\kern\wd0}

\newcommand{\N}{\mathbb{N}}

\newcommand{\Cfrak}{\ensuremath{\mathfrak{C}}}

\author[I. Eleftheriadis]{Ioannis {Eleftheriadis}\ \orcidlink{0000-0003-4764-8894}}
\address{Department of Computer Science and Technology\\
University of Cambridge\\15 JJ Thomson Ave, CB3 0FD\\Cambridge, UK}
\urladdr{https://www.cst.cam.ac.uk/people/ie257}
\email{ie257@cam.ac.uk}

\thanks{Supported by a George and Marrie Vergottis Scholarship awarded through Cambridge Trust, an Onassis Foundation Scholarship, and a Robert Sansom Studentship.}



\title{Universal slices of the category of graphs}

\begin{document}
\maketitle

\begin{abstract}
        We characterise the slices of the category of graphs that are algebraically universal in terms of the structure of the slicing graph. In particular, we show that algebraic universality is obtained if, and only if, the slicing graph contains one of four fixed graphs as a subgraph. 
\end{abstract}

\section{Introduction}

The notion of \emph{algebraic universality} was defined by Isbell \cite{isbell} as a generalisation to a series of representation results in different categories of structures \cite{birkhoff},\cite{groot},\cite{sabidussi}. The various degrees of representability in a category are captured by the following definition. 

\begin{definition*}[\cite{setendofunctors}]
    We say that a category $\Cfrak$ is:
    \begin{itemize}
        \item group-universal, if for every group $G$ there is some $C \in \obj(\Cfrak)$ such that $\Aut_\Cfrak(C)$ is isomorphic to $G$;
        \item strongly group-universal, if for every group $G$ there is some $C \in \obj(\Cfrak)$ such that $\End_\Cfrak(C)$ is a group isomorphic to $G$;
        \item monoid-universal, if for every monoid $M$ there is some $C \in \obj(\Cfrak)$ such that $\End_\Cfrak(C)$ is isomorphic to $M$;
        \item alg-universal, if every algebraic category fully embeds into $\Cfrak$;
        \item universal, if every concretisable category fully embeds into $\Cfrak$;
        \item hyper-universal, if every category fully embeds into $\Cfrak$.
    \end{itemize}
\end{definition*}

We briefly discuss the relationship between these notions. By a result of Pultr \cite{pultr1964concerning} it follows that algebraically universal categories fully embed all small categories. Viewing any monoid as a one-object category, this establishes that algebraic universality implies monoid-universality. Moreover, by work of Kučera \cite{kucera}, Hedrlin, and Pultr \cite{hedrlinpultr}, it follows under set-theoretic assumptions that every concretisable category is algebraic, i.e. isomorphic to a full category of universal algebras, and so universality coincides with algebraic universality. For all remaining notions the implications are strict, though the examples witnessing this are not always natural. 

The study of algebraically universal categories was extensively carried out by the ``Prague school'' of category theory in the 1960s (see \cite{pultr} for an overview). In addition to the results above, it was established \cite{hedrlin1966full} that the category of graphs with graph homomorphisms is algebraically universal. At the same time, this category is itself fully embeddable in the category of binary universal algebras; consequently, algebraic universality amounts to the existence of a full embedding from the category of graphs.

This immediately raises interesting combinatorial questions: \emph{which ``nicely-defined'' classes of graphs produce algebraically universal categories?} Several different approaches have been taken to questions of this style \cite{dirgra}, \cite{nesetrilossona}, \cite{algunirel}. Here, we consider one that is more natural to the working category theorist: \emph{which graphs produce algebraically universal categories when slicing over them?} We are able to provide a complete description of when this happens. 

\begin{theorem*}
    The following are equivalent for an undirected graph $G$:
    \begin{enumerate}
        \item\label{1} $G$ is not a subgraph of a disjoint union of paths of length three;
        \item\label{2} $G$ contains one of $C_3, C_4, P_4$, or $Y$ as a subgraph; 
        \item\label{3} The slice category $\Gra/G$ is algebraically universal. 
    \end{enumerate}
    
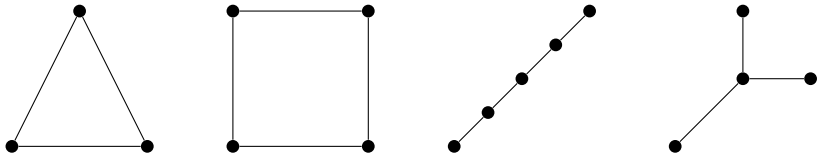
\begin{figure}[h!]
    \begin{tikzpicture}[scale=0.30]

  \node[fill,circle,scale=0.5] (A) at (0, 0) {};
  \node[fill,circle,scale=0.5]  (B) at (6, 0) {};
  \node[fill,circle,scale=0.5]  (C) at (3, 6) {};

  \draw(A) -- (B);
  \draw (B) -- (C);
  \draw (A) -- (C);

\end{tikzpicture}
    \qquad
    \begin{tikzpicture}[scale=0.30]

  \node[fill,circle,scale=0.5] (A) at (0, 0) {};
  \node[fill,circle,scale=0.5]  (B) at (6, 0) {};
  \node[fill,circle,scale=0.5]  (C) at (6, 6) {};
  \node[fill,circle,scale=0.5]  (D) at (0, 6) {};

  \draw(A) -- (B);
  \draw (B) -- (C);
  \draw (C) -- (D);
  \draw (A) -- (D);

\end{tikzpicture}
    \qquad
    \begin{tikzpicture}[scale=0.30]

  \node[fill,circle,scale=0.5] (A) at (0, 0) {};
  \node[fill,circle,scale=0.5]  (B) at (1.5, 1.5) {};
  \node[fill,circle,scale=0.5]  (C) at (3, 3) {};
  \node[fill,circle,scale=0.5]  (D) at (4.5, 4.5) {};
  \node[fill,circle,scale=0.5]  (E) at (6, 6) {};

  \draw(A) -- (B);
  \draw (B) -- (C);
  \draw (C) -- (D);
  \draw (E) -- (D);

\end{tikzpicture}
    \qquad
    \begin{tikzpicture}[scale=0.30]

  \node[fill,circle,scale=0.5] (A) at (0, 0) {};
  \node[fill,circle,scale=0.5]  (B) at (3, 3) {};
  \node[fill,circle,scale=0.5]  (C) at (3, 6) {};
  \node[fill,circle,scale=0.5]  (D) at (6, 3) {};

  \draw(A) -- (B);
  \draw (B) -- (C);
  \draw (B) -- (D);

\end{tikzpicture}
\end{figure}
\end{theorem*}

The equivalence of \ref{1} and \ref{2} is trivial. In the case that $G$ contains one of the specified graphs we produce a fully faithful functor $\Gra \to \Gra/G$ witnessing universality\footnote{The embedding is even  \emph{strong} in the sense of \cite{pultr}, I.6.11}. This is established via an extension of the arrow construction of graphs to slice categories. On the other hand, whenever $G$ is a disjoint union of paths of length three we show that $\Gra/G$ is not even strongly group-universal by proving that every object in $\Gra/G$ is either rigid or it has a proper endomorphism, i.e. an endomorphism which is not an automorphism.

\section{The category of graphs}

Unless specified otherwise, graphs always refer to undirected graphs. A graph homomorphism is a map $f: G \to H$ such that 
\[ (u,v) \in E(G) \implies (f(u),f(v)) \in E(H).\]
We write $\Gra$ for the category of graphs with graph homomorphisms. By a digraph we mean a set $D$ with an arbitrary binary relation $E(D)\subseteq D\times D$. Digraph homomorphisms are defined analogously, and we write $\Rel(2)$ for the category of digraphs with homomorphisms. An isolated point of a digraph $D$ is an element $x \in D$ which does not appear in any edge of $D$. We write $\Rel(2)^+$ for the full subcategory of $\Rel(2)$ on the digraphs without isolated points. 

\begin{fact}[\cite{pultr}]
    The categories $\Gra, \Rel(2), \Rel(2)^+$ are algebraic and algebraically universal. 
\end{fact}

In most examples, algebraic universality is established via a particular type of fully faithful functor from the category of digraphs. This is often referred to as the \emph{arrow construction} \cite{mendelsohn}. Intuitively, the idea is to glue a copy of a fixed ``arrow'' taken from the target category to each edge of a given digraph. Provided that the arrow is chosen appropriately, this assignment gives rise to a full embedding. We describe the construction in detail in the case of graphs; this will then be extended to the context of slice categories in \Cref{sec:arrowslice}. 

\subsection{The arrow construction for graphs}

\begin{definition}
    We call a tuple $(H,a,b)$ where $H$ is an undirected graph and $a,b$ are two vertices of $H$ a \emph{graph gadget with respect to $a$ and $b$}. Whenever it is clear from the context, we simply write $H$ for $(H,a,b)$. Given a graph gadget $(H,a,b)$ and a digraph $D$ we define an undirected graph $D \star H$ on the domain $D \sqcup (E(D)\times H\setminus\{a,b\})$. For every edge $(u,v)$ of $D$ define the maps $\phi_{(u,v)}^{D,H}:H \to D \star H$ as follows
    \begin{equation*}
    \phi_{(u,v)}^{D,H}(x) =
    \begin{cases*}
      u, & if $x=a$; \\
      v, & if $x=b$; \\
      (u,v,x), & otherwise.
    \end{cases*}
  \end{equation*}
  We then define the edge set of $D \star H$ as:
  \[\{(x,y): \exists (u,v) \in E(D), (s,t) \in E(H) \text{ such that } \phi_{(u,v)}^{D,H}(s,t)=(x,y)\}.\]
  For a homomorphism $f:D_1 \to D_2$ we define the map $f\star H: D_1 \star H \to D_2 \star H$ by letting
    \begin{equation*}
    f \star H(x) =
    \begin{cases*}
      f(x), &if $x \in D_1$; \\
      (u,v,f(w)), &if $x=(u,v,w)$ for $(u,v)\in E(D_1)$ and $w \in H\setminus\{a,b\}$.
    \end{cases*}
  \end{equation*}
\end{definition}

\begin{lemma}[\cite{pultr},\cite{algunirel}]
    For every graph gadget $(H,a,b)$ and digraph $D$ the maps $\phi_{(u,v)}^{D,H}$ are injective strong homomorphisms. Moreover, for a digraph homomorphism $f:D_1 \to D_2$ the map $f\star H: D_1 \star H \to D_2 \star H$ is a graph homomorphism. Finally, the assignment $F: \Rel(2)^+ \to \Gra$ sending $D \mapsto D\star H$ and $f \mapsto f\star H$ is a faithful functor.
\end{lemma}

The following definition and theorem illustrate the effectiveness of the arrow construction in the context of undirected graphs. 

\begin{definition}
    We say that an undirected graph $H$ is a \emph{strong replacement graph} with respect to two points $a,b$ if for all irreflexive digraphs $D$ and all homomorphisms $f:H \to D \star H$ the homomorphic image $f[H]$ is contained in some copy of $H$ in $D \star H$, i.e. some $\phi_{(u,v)}^{D,H}[H]$ for $(u,v) \in E(D)$. 
\end{definition}

\begin{theorem}[\cite{neshell}]\label{thm:strongreplace}
    Fix a full subcategory $\Cfrak$ of $\Gra$, and let $H$ be a rigid undirected graph which is a strong replacement graph with respect to $a,b \in H$. Suppose that for every digraph $D$ without isolated points $D \star (H,a,b)$ is in $\Cfrak$. Then $\Cfrak$ is algebraically universal. 
\end{theorem}

As an application of the above, we show that the full subcategory $\Gra \to G$ of $\Gra$ consisting of those graphs that homomorphically map to $G$ is universal if, and only if, $G$ is non-bipartite. This result is folklore; however, to the author's best understanding it has not been explicitly written down. We provide a short proof as it is of similar flavour to the main theorem of this paper. One direction is trivial: if $G$ is bipartite then $\Gra \to G$ consists of bipartite graphs, and since every bipartite graph is either rigid or it admits a proper endomorphism, it is not even strongly group-universal. For the converse, we use an appropriate graph gadget.

\begin{definition}
    The undirected graph $G_k$ is defined as in the figure below. Dotted lines denote paths of length $k$. The labels denote the image of the respective points under a homomorphism to the cycle $C_{2k-1}$, where $a$ and $b$ are mapped to $k$.
    \begin{figure}[h!]
    \begin{tikzpicture}[scale=0.32]

  \node[fill,circle,scale=0.5,label={0}] (A) at (0, 0) {};
  \node[fill,circle,scale=0.5,label={k}]  (B) at (6, 0) {};
  \node[fill,circle,scale=0.5,label={1}]  (C) at (12, 0) {};
  \node[fill,circle,scale=0.5,label=below:1]  (D) at (6, 6) {};
  \node[fill,circle,scale=0.5,label=below:{1}]  (E) at (2, 2) {};
  \node[fill,circle,scale=0.5, label=below:{0}]  (F) at (4,4) {};
  \node[fill,circle,scale=0.5, label=below:{0}]  (H) at (9,3) {};
  \node[fill,circle,scale=0.5,label={1}]  (I) at (2,4) {};
  \node[fill,circle,scale=0.5,label={a}]  (K) at (4,8) {};
  \node[fill,circle,scale=0.5,label={0}]  (L) at (6,12) {};
  \node[fill,circle,scale=0.5,label={b}]  (M) at (7.5,9) {};
  \node[fill,circle,scale=0.5, label={1}]  (N) at (9,6) {};
  \node[fill,circle,scale=0.5,label={0}]  (P) at (10.5,3) {};

  \draw[dotted] (A) -- (B);
  \draw[dotted] (B) -- (C);
  \draw (A) -- (E);
  \draw (E) -- (F);
  \draw (F) -- (D);
  \draw (H) -- (D);
  \draw (C) -- (H);
  \draw (A) -- (I);
  \draw[dotted] (I) -- (K);
  \draw[dotted] (K) -- (L);
  \draw (L) -- (D);
  \draw[dotted] (L) -- (M);
  \draw[dotted] (M) -- (N);
  \draw (N) -- (P);
  \draw (P) -- (C);

\end{tikzpicture}
    \label{fig:Gk}
\end{figure}
\end{definition}

\begin{fact}\cite{neshell}\label{factrigid}
    The graph $G_k$ is rigid, and a strong replacement graph with respect to the points $a$ and $b$. 
\end{fact}

\begin{theorem}
    If $G$ is non-bipartite, then $\Gra \to G$ is algebraically universal. 
\end{theorem}

\begin{proof}
    Clearly, if $G$ is non-bipartite then it contains a cycle of length $2k-1$ for some $k \in \N$. Hence, it suffices to show that $\Gra \to C_{2k-1}$ is algebraically universal. Consider the graph $G_k$ from \Cref{fig:Gk}. By \Cref{factrigid}, $G_k$ is rigid and a strong replacement graph with respect to the points $a$ and $b$. Moreover, \Cref{fig:Gk} exhibits a homomorphism from $G_k$ to $C_{2k-1}$ that identifies $a$ and $b$. Consequently, for every irreflexive graph $D$ without isolated points there is a homomorphism $D \star G_k \to C_{2k-1}$. \Cref{thm:strongreplace} therefore implies that $\Gra \to C_{2k-1}$ is algebraically universal. 
\end{proof}

\begin{corollary}
    The following are equivalent for an undirected graph $G$:
    \begin{enumerate}
        \item $G$ is not bipartite;
        \item $G$ contains a cycle of odd length;
        \item $\Gra \to G$ is algebraically universal.
    \end{enumerate}
\end{corollary}

Note that despite the close relationship of this result to our \Cref{thm:main}, the characterisation in our context is very different: indeed slices over bipartite graphs can still be alg-universal. 

\section{Slice categories of graphs}

Recall the definition of the slice category over an object. 

\begin{definition}
    Let $\Cfrak$ be a category, and $C \in \obj(\Cfrak)$. The \emph{slice category} or \emph{over category} $\Cfrak/C$ of $\Cfrak$ over $C$ consists of
    \begin{itemize}
        \item pairs $(A,f)$ where $A \in \obj(\Cfrak)$ and $f:A \to C \in \morph(\Cfrak)$ as objects;
        \item morphisms $\psi:A \to B \in \morph(\Cfrak)$ such that $f = g \circ \psi$ as morphisms $(A,f)\to(B,g)$.
    \end{itemize}
    \begin{figure}[h!]
    \begin{tikzcd}[scale=0.4]
A \arrow[dd, "\psi"'] \arrow[rr, "f"] &  & C \\
                                        &  &   \\
B \arrow[rruu, "g"']                    &  &  
\end{tikzcd}
    \end{figure}
\end{definition}

Note that is is not immediately obvious that $\Cfrak/C$ is algebraic assuming that $\Cfrak$ is. We provide the argument for graphs, which can be easily adapted to the general case.

\begin{lemma}
    For all graphs $G$ the category $\Gra/G$ is algebraic. 
\end{lemma}

\begin{proof}
    Let $G=\{v_i : i < \kappa\}$ be an enumeration of the vertices of $G$. For each $i <\kappa$ let $P_i$ be a unary relation symbol, and consider the category $\Rel(2,(P_i)_{i<\kappa})$ of digraphs expanded with $\kappa$ unary predicates; this is algebraic. Define a functor $F:\Gra/G \to \Rel(2,(P_i)_{i<\kappa})$ by mapping $(H,f)$ to the structure $H_f$ whose binary relation is interpreted in the same way as the edge relation of $H$, while for each $i<\kappa$ the unary predicate $P_i$ is interpreted as $f^{-1}(v_i)\subseteq H$. It is then clear that the homomorphisms $(H,f)\to(D,g) \in \morph(\Gra/G)$ are in one-to-one correspondence with the homomorphisms $H_f \to D_g \in \morph(\Rel(2,(P_i)_{i<\kappa}))$. It follows that $F$ is fully faithful and hence $\Gra/G$ is algebraic. 
\end{proof}

We proceed to describe the arrow construction in the context of slice categories. The idea is very much similar to the graph case, only here the arrow is an object $(H,f)$ from $\Gra/G$. We fix two points $a,b \in H$ and glue $H$ to each edge $(u,v)$ of a given digraph $D$ by identifying $u$ with $a$ and $v$ with $b$. To ensure that the resulting object does in fact map to $G$ for every choice of $D$, we require that $a$ and $b$ are identified under $f$. 

\subsection{The arrow construction for slice categories}\label{sec:arrowslice}

    Let $H,G$ be undirected graphs, $f:H \to G$ a homomorphism, and $a,b$ two vertices of $H$ such that $f(a)=f(b)$. For every digraph $D$ define the map $f_D: D\star (H,a,b) \to G$ by letting 
    
    \begin{equation*}
    f_D(x) =
    \begin{cases*}
      f(a), & if $x \in D$; \\
      f(y), & if $x = (u,v,y)$ for some $(u,v)\in E(D)$ and $y \in H$.
    \end{cases*}
  \end{equation*}

  We claim that this is a homomorphism to $G$. Indeed, if $(s,t) \in E(D \star H)$ then there is an edge $(u,v) \in E(D)$ and an edge $(x,y) \in H$ such that $\phi_{(u,v)}(x)=s$ and $\phi_{(u,v)}(y)=t$. We distinguish cases. If $s,t \in D$ then there is an edge $(a,b)$ in $H$ implying that there is an edge $f((a),f(b))$ in $G$, contradicting that $f(a)=f(b)$. If $s \in D$ and $t = (u,v,w)$ then either there is an edge $(a,w)$ or an edge $(b,w)$ in $H$; in either case, since $f(a)=f(b)=f_D(s)$ and $f_D(t)=f(w)$ there is an edge $(f_D(s),f_D(t))$ in $G$. The case where $s = (u,v,w)$ and $t \in D$ is handled symmetrically. Finally, if $s=(u,v,w)$ and $t=(u,v,z)$ then $(w,z)$ is an edge in $H$; it follows that $(f(w),f(z))=(f_D(s),f_D(t))$ is an edge in $G$. Consequently, $f_D$ is a valid homomorphism $D \star H \to G$, and therefore $(D \star H,f_D)$ is an object of $\Gra/G$.

  We also show how to extend the arrow construction to morphisms in the slice category. Let $h: D_1 \to D_2$ be a homomorphism of directed graphs, and $H,f, a, b$ as above. Consider the morphism $h \star H: D_1 \star H \to D_2 \star H$. We argue that this extends to a map $h \star H: (D_1 \star H,f_{D_1}) \to (D_2 \star H,f_{D_2})$. Indeed, if $x \in D_1$ then $f_{D_2}\circ (h \star H)(x)=f_{D_2}(y)$ where $y \in D_2$ and therefore this is equal to $f(a)$, which in turn is equal to $f_{D_1}(x)$. Moreover, if $x = (u,v,w)$ for $(u,v)\in E(D_1)$ and $w \in H$ then $f_{D_2}\circ (h \star H)(x)=f_{D_2}((h(u),h(v)),w)=f(w)=f_{D_1}(x)$. It follows that $f_{D_1}=f_{D_2}\circ (h \star H)$ as required. 

  Finally, observe that the maps $\phi_{(u,v)}^{D,H}:H \to D \star H$ naturally extend to maps $\phi_{(u,v)}^{D,H}: (H,f) \to (D\star H,f_D)$ in $\Gra/G$. Indeed, 
  \[f_D \circ \phi_{(u,v)}^{D,H}(a)=f_D(u)=f(a);\]
  \[f_D \circ \phi_{(u,v)}^{D,H}(b)=f_D(v)=f(a)=f(b);\]
  \[f_D \circ \phi_{(u,v)}^{D,H}(w)=f_D(u,v,w)=f(w), \text{for }w\neq a,b,\]
  
  and therefore $f_D \circ \phi_{(u,v)}=f$ as required. 

Once again, the functionality of our construction is verified by its power to transfer algebraic universality under assumptions on the arrow. 

  \begin{theorem}\label{thm:arrow}
      Let $H,G$ be an undirected graph, $f:H \to G$ a homomorphism and $a,b$ two vertices of $H$ such that $f(a)=f(b)$. Suppose that for every digraph $D$ without isolated points the only homomorphisms $(H,f)\to (D \star H,f_D)$ in $\Gra/G$ are the maps $\phi_{(u,v)}^{D,H}$. Then $\Gra/G$ is algebraically universal.  
  \end{theorem}

  \begin{proof}
      We argue that the assignment $F:\Rel(2)^+ \to \Gra/G$ given by
    \begin{align*}
        D &\mapsto (D \star H,f_D) \\
        h &\mapsto h \star H
    \end{align*}
    is a fully faithful functor. Since $\Rel(2)^+$ is algebraically universal, this implies that $\Gra/G$ is algebraically universal as required. 

    It is clear that $F$ is functorial, so we first show faithfullness. Suppose that $h:D_1 \to D_2, g: D_1 \to D_2$ are two distinct homomorphisms, i.e. there is some $x \in D_1$ such that $h(x)\neq g(x)$. It follows that $h \star H(x)=h(x)\neq g(x)=g\star H(x)$, and therefore $h \star H \neq g \star H$.

    We also argue for fullness. Let $\chi:(D_1 \star H,f_{D_1}) \to (D_2\star H,f_{D_2})$ be a morphism in $\Gra/G$. It follows by the assumption that for every edge $(u,v) \in E(D_1)$ there is an edge $(u',v') \in E(D_2)$ such that $\chi\circ \phi_{(u,v)}=\phi_{(u',v')}$. In particular, we obtain that 
    \[ u' = \phi_{(u',v')}(a) = \chi\circ \phi_{(u,v)}(a) = \chi(u), \text{ and}\]
    \[ v' = \phi_{(u',v')}(b) = \chi\circ \phi_{(u,v)}(b) = \chi(v).\]
    Moreover, since $D_1$ has no isolated points it follows that every $u \in D_1$ appears in some edge. Hence the map $g:=\chi_{\restriction D_1}$ is a well-defined homomorphism $D_1 \to D_2$. We argue that $\chi = g \star H$. It is clear by construction that $\chi$ and $g \star H$ agree on $D_1$, so let $x = (u,v,w) \in D_1$ for $(u,v) \in E(D_1)$ and $w \in H\setminus\{a,b\}$. Then:
    \[ \chi(x)=\chi \circ \phi_{(u,v)}(w) = \phi_{(g(u),g(v))}(w) = g \star H(x),\]
    implying that $\chi = g \star H$ as required. 
  \end{proof}

\section{Main theorem}

The machinery developed above allow us to prove our main theorem, which we state once more below. 
\begin{theorem}\label{thm:main}
    The following are equivalent for an undirected graph $G$:
    \begin{enumerate}
        \item $G$ is not a subgraph of a disjoint union of paths of length three;
        \item $G$ contains one of $C_3, C_4, P_4$, or $Y$ as a subgraph; 
        \item The slice category $\Gra/G$ is algebraically universal. 
    \end{enumerate}
\end{theorem}

The proof is divided in two parts. For non-universality, we establish that if $G$ is a subgraph of a disjoint union of paths of length three then every object in $\Gra/G$ is either rigid or it has a proper endomorphism. The idea is to first show this for a single copy of $P_3$ and for objects $(H,f)$ such that $H$ is connected and $f$ is surjective. We then deduce the general case from this. For universality, we reduce the problem to the case when $G$ is one of the four specified graphs, and in each case describe an appropriate gadget that satisfies the conditions of \Cref{thm:arrow}. 

\subsection{Non-universality}

\begin{lemma}\label{lem:rigidslice}
    Let $G$ be a graph, and $(H,f)$ an object in $\Gra/G$. Assume that for every connected component $A$ of $H$ the object $(A,f_{\restriction A})$ is rigid in $\Gra/G$, while for every connected component $B\neq A$ it holds that $f[A]\not\subseteq f[B]$ and $f[B] \not\subseteq f[A]$. Then $(H,f)$ is rigid. 
\end{lemma}

\begin{proof}
    Assume for a contradiction that $\phi:(H,f) \to (H,f)$ has a non-trivial endomorphism, i.e. one such that there are $a\neq b$ from $H$ satisfying $\phi(a)=b$. It follows that $a$ and $b$ lie in different connected components of $H$; indeed, if they were in the same component $A$ then the map $\phi_{\restriction A}$ would be a non-trivial endomorphism of $(A,f_{\restriction A})$, contradicting its rigidity. So, there are disjoint connected components $A,B$ of $H$ such that $a \in A$ and $b \in B$. Since homomorphisms preserve connectedness, it follows that $\phi[A]\subseteq \phi[B]$. Consequently, $f[A] \subseteq f[B]$, contradicting the assumption. 
\end{proof}

\begin{lemma}\label{lem:pathslice}
    Fix $n \in \{0,1,2,3\}$. Let $H$ be a connected graph and $f:H \to P_n$ a surjective homomorphism. Then either $(H,f)$ is rigid in $\Gra/P_n$ or it has a proper endomorphism. 
\end{lemma}

\begin{proof}
    The cases $n=0,1,2$ follow by the observation that, for such $n$, any connected graph $H$ surjectively mapping to $P_n$ via some homomorphism $f$ contains a connected subgraph $A$ such that the restriction of $f$ on $A$ is an isomorphism. In particular, it follows that $(H,f)$ homomorphically maps to $(A,f_{\restriction A})$. Therefore, if $A$ is a proper substructure this gives a proper endomorphism of $(H,f)$; otherwise $A=H$ and therefore $(H,f)$ is rigid since $f$ is in fact an isomorphism. 
    
    We proceed to show this for $n=3$. Write $\{0,1,2,3\}$ for the vertices of $P_3$. Since $f$ is surjective, it follows that there is a path $P=u_0,\dots,u_k$ in $H$ of shortest length such that $f(u_0)=0$ and $f(u_k)=3$. Minimality ensures that $k=2c+3$ for some $c \in \N$ and moreover
      \begin{equation}\tag{$*$}\label{eq:*}
    f(u_i) =
    \begin{cases*}
      0, & if $i=0$; \\
      1, & if $i \in [k-1]$ is odd; \\
      2, & if $i \in [k-1]$ is even; \\
      3, & if $i=k$.
    \end{cases*}
  \end{equation}
    
    We argue that there is a homomorphism $(H,f) \to (P,f_{\restriction P})$. Indeed, let $\tau:H \to \N$ be the map 
    \[ \tau(z) = \min\{\text{dist}_H(z,w) : w \in H, f(w)=0\}\]
    
    and define $\phi: H \to P$ by letting $\phi(x)$ be the unique vertex $u_i$ of $P$ with $f(x)=f(u_i)$ and such that $|\tau(x)-i|$ is minimal. 

    It is clear that $\phi$ is the identity on $P$, while $f = f \circ \phi$. Moreover, $\phi$ is a graph homomorphism: let $(x,y)$ be an edge in $H$, and assume that $\phi(x)=u_i$. Since $(f(x),f(y))$ is an edge in $P_3$, it follows that either $\tau(y)=\tau(x)+1$ or $\tau(x)=\tau(y)-1$. In the former case $\phi(y)=u_{i+1}$ since $|\tau(y)-(i+1)|=|\tau(x)-i|$ is minimal, while in the latter $\phi(y)=u_{i-1}$ since $|\tau(y)-(i-1)|=|\tau(x)-i|$ is minimal. In either case, $(\phi(x),\phi(y))$ is an edge in $P$ as required. 

    It follows that $\phi$ is a retract of $(H,f)$ on $(P,f_{\restriction P})$. This implies that if $P$ is a proper substructure of $H$ then $H$ has proper endomorphism, while if $H=P$ minimality ensures that $H$ is rigid. 
\end{proof}

The proof above also illustrates that any two connected objects in $\Gra/P_3$ which have $\subseteq$-comparable images are homomorphism-comparable. 

\begin{corollary}\label{cor:comparable}
    Let $(A,f),(B,g)$ be connected objects in $\Gra/P_3$ such that $f[A]\subseteq g[B]$. Then there is either a homomorphism $(A,f) \to (B,g)$ or a homomorphism $(B,g) \to (A,f)$. 
\end{corollary}

\begin{proof}
    Let $(A,f),(B,g)$ be as above. We distinguish two cases. If $f[A]=f[B]=P_3$ then by the argument above $(A,f)$ and $(B,g)$ retract onto paths $(P_A,f_{\restriction P_A})$ and $(P_B,g_{\restriction P_B})$ respectively, satisfying (\ref{eq:*}). Clearly, any path of this form homomorphically maps to a longer one also of this form. It follows that either $(P_A,f_{\restriction P_A})\to(P_B,g_{\restriction P_B})$ or $(P_B,g_{\restriction P_B})\to(P_A,f_{\restriction P_A})$ and therefore $(A,f) \to (B,g)$ or $(B,g) \to (A,f)$ by composing homomorphisms. 

    On the other hand if $f[A]\subsetneq P_3$ then $f[A]$ is isomorphic to $P_n$ for $n \in \{0,1,2\}$. Therefore, since $g[B]\supseteq f[A]$, there is some connected $H\subseteq B$ on which $g$ is an isomorphism with $f[A]$. It follows that $(g_{\restriction H})^{-1}\circ f$ is a homomorphism from $(A,f)$ to $(B,g)$.
\end{proof}

\begin{theorem}\label{thm:nonuni}
    Let $G$ be a subgraph of $\coprod_{\kappa} P_3$ for some cardinal $\kappa$. Then $\Gra/G$ is not algebraically universal. 
\end{theorem}

\begin{proof}
    Let $G$ be as above. We argue that every $(H,f)$ from $\Gra/G$ either has a proper endomorphism or it is rigid; this implies that $\Gra/G$ is not even strongly group-universal. 

    Fix $(H,f) \in \obj(\Gra/G)$, and assume that every endomorphism of $(H,f)$ is an automorphism. Write $H = \coprod_{i<\lambda} H_i$ into a disjoint union of its connected components and let $f_i = f_{\restriction{H_i}}$ be the restriction of $f$ on $H_i$. We first argue that each object $(H_i,f_i)$ is rigid. Since $H_i$ is connected, it follows that $f_i[H_i]$ is connected, and so it is one of $P_n$ for $n \in \{0,1,2,3\}$. Consequently, \Cref{lem:pathslice} implies that $(H_i,f_i)$ is either rigid or it has a proper endomorphism. Since any proper endomorphism of $(H_i,f_i)$ may be extended to a proper endomorphism of $(H,f)$, our assumption implies that $(H_i,f_i)$ is indeed rigid. 

    We additionally argue that $f[H_i]\not\subseteq f[H_j]$ for all $i \neq j$. Indeed, assume for a contradiction that there are $i\neq j$ from $\lambda$ such that $f[H_i] \subseteq f[H_j]$. Since $H_i,H_j$ are connected their image is contained in the same copy of $P_3$, and so it follows by \Cref{cor:comparable} that there is either a homomorphism $(H_i,f_i) \to (H_j,f_j)$ or a homomorphism $(H_j,f_j) \to (H_i,f_i)$. In either case, by extending with the identity homomorphism, this gives rise to a proper endomorphism of $(H,f)$ which is a contradiction. Consequently, it follows by \Cref{lem:rigidslice} that $(H,f)$ is rigid as required. 
\end{proof}

\subsection{Universality} 

\begin{lemma}\label{lem:c3}
    $\Gra/C_3$ is algebraically universal.
\end{lemma}

\begin{proof}
    Write $\{0,1,2\}$ for the set of vertices of $C_3$ and let $P_3$ be the path of length three with vertex set $\{a,b,c,d\}$. Define the homomorphism $f: P_3 \to C_3$ as
    \begin{multicols}{4}
    \begin{itemize}[label={}]
        \item $a\mapsto 0$
        \item $b \mapsto 1$
        \item $c \mapsto 2$
        \item $d \mapsto 0$.
    \end{itemize}
    \end{multicols}
    Since $f(a)=f(d)$, we may view $(P_3,f)$ as a gadget with respect to $a$ and $d$. We argue that for every directed graph $D$ without isolated points the only homomorphisms $(P_3,f)\to (D\star P_3,f_D)$ are the maps $\phi_{(u,v)}^{D,P_3}$; \Cref{thm:arrow} then implies that $\Gra/C_3$ is algebraically universal. 

    So, let $\chi:(P_3,f) \to (D \star P_3,f_D)$ be a morphism in $\Gra/C_3$. Since $f = f_D \circ \chi$ and $b$ is the only element sent to $1$ under $f$, it follows that $\chi(b)=(u,v,b)$ for some edge $(u,v)$ in $D$. From this, it is easy to see that 
    \[\chi(a)=u, \chi(c)=(u,v,c), \chi(d)=v,\]
    and therefore $\chi = \phi_{(u,v)}^{D,P_3}$ as claimed. 
\end{proof}

\begin{lemma}
    $\Gra/C_4$ is algebraically universal.
\end{lemma}

\begin{proof}
    Write $\{0,1,2,3\}$ for the vertex set of $C_4$, and $\{a,b,c,d,e\}$ for the vertex set of $P_4$. We consider the homomorphism $f:P_4 \to C_4$ defined by
    \begin{multicols}{5}\begin{itemize}[label={}]
        \item $a\mapsto 0$
        \item $b \mapsto 1$
        \item $c \mapsto 2$
        \item $d \mapsto 3$
        \item $e \mapsto 0$
    \end{itemize}\end{multicols}
    The remaining argument proceeds as \Cref{lem:c3} by using $(P_4,f)$ as a gadget with respect to $a$ and $e$.
\end{proof}

\begin{lemma}
    $\Gra/P_4$ is algebraically universal.
\end{lemma}

\begin{proof}
    Write $\{0,1,2,3,4\}$ for the vertices of $P_4$, and $\{a,\dots,m\}$ for the vertex set of $P_{12}$. Consider the homomorphism $\psi:P_{12} \to P_4$ defined by: 
    \begin{multicols}{5}
    \begin{itemize}[label={}]
        \item $a\mapsto 0$
        \item $b \mapsto 1$
        \item $c \mapsto 2$
        \item $d \mapsto 1$
        \item $e \mapsto 2$
        \item $f \mapsto 3$
        \item $g \mapsto 4$
        \item $h \mapsto 3$
        \item $i \mapsto 2$
        \item $j \mapsto 3$
        \item $k \mapsto 2$
        \item $l \mapsto 1$
        \item $m \mapsto 0$
    \end{itemize}
    \end{multicols}
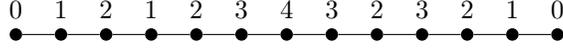
    \begin{figure}[h!]
        \begin{tikzpicture}[scale=0.30]

  \node[fill,circle,scale=0.5,label={0}] (A) at (0, 0) {};
  \node[fill,circle,scale=0.5,label={1}]  (B) at (2, 0) {};
  \node[fill,circle,scale=0.5,label={2}]  (C) at (4, 0) {};
  \node[fill,circle,scale=0.5,label={1}]  (D) at (6, 0) {};
  \node[fill,circle,scale=0.5,label={2}]  (E) at (8, 0) {};
  \node[fill,circle,scale=0.5,label={3}]  (F) at (10, 0) {};
  \node[fill,circle,scale=0.5,label={4}]  (G) at (12, 0) {};
  \node[fill,circle,scale=0.5,label={3}]  (H) at (14, 0) {};
  \node[fill,circle,scale=0.5,label={2}]  (I) at (16, 0) {};
  \node[fill,circle,scale=0.5,label={3}]  (J) at (18, 0) {};
  \node[fill,circle,scale=0.5,label={2}]  (K) at (20, 0) {};
  \node[fill,circle,scale=0.5,label={1}]  (L) at (22, 0) {};
  \node[fill,circle,scale=0.5,label={0}]  (M) at (24, 0) {};

  \draw(A) -- (B);
  \draw (B) -- (C);
  \draw (C) -- (D);
  \draw (E) -- (D);
  \draw (E) -- (F);
  \draw (F) -- (G);
  \draw (H) -- (G);
  \draw (I) -- (H);
  \draw (J) -- (I);
  \draw (K) -- (J);
  \draw (L) -- (K);
  \draw (M) -- (L);

\end{tikzpicture}
    \end{figure}
    Since $\psi(a)=\psi(m)=0$, we consider $(P_{12},\psi)$ as a gadget with respect to $a$ and $m$. As above, we argue that for every directed graph $D$ without isolated points the only homomorphisms $(P_{12},\psi)\to (D\star P_{12},\psi_D)$ in $\Gra/P_4$ are the maps $\phi_{(u,v)}^{D,P_{12}}$ for $(u,v) \in E(D)$; this shows universality by \Cref{thm:arrow}.

    Letting $\chi:(P_{12},\psi)\to (D\star P_{12},\psi_D)$ be a homomorphism we see that $\chi(g)=(u,v,g)$ for some edge $(u,v) \in E(D)$, since $\psi= \psi_D \circ \chi$ and $g$ is the only element of $P_{12}$ that is mapped to $4$ under $\psi$. Consider $\chi(f)$: since $(f,g) \in E(P_{12})$ it follows that this is either $(u,v,f)$ or $(u,v,h)$. We argue that the latter leads to a contradiction. Indeed, if $\chi(f)=(u,v,h)$ then necessarily $\chi(e)=(u,v,i)$ since that is the only vertex adjacent to $(u,v,h)$ whose image under $\psi_D$ is $2$. But then $\chi(d)$ ought to be one of $(u,v,h)$ or $(u,v,j)$; however this violates $\psi= \psi_D \circ \chi$. Consequently, $\chi(f)=(u,v,f)$. This forces $\chi(a)=u$ as $u$ is the only element of distance $\leq 5$ from $(u,v,f)$ in $D \star P_{12}$ with $\psi_D(u)=0$. It therefore also follows that $\chi(x)=(u,v,x), \text{ for } x\in \{b,c,d,e\}$. A similar argument ensures that $\chi(h)=(u,v,h)$, which in turn implies that $\chi(m)=v$, and so $\chi(x)=(u,v,x), \text{ for } x \in \{i,j,k,l\}$. Hence $\chi = \phi_{u,v}^{D,P_{12}}$ as claimed. 
\end{proof}

\begin{lemma}
    $\Gra/Y$ is algebraically universal. 
\end{lemma}

\begin{proof}
    Write $\{0,1,2,3\}$ for the vertices of $Y$ as in the figure below, and $\{a,\dots,g\}$ for the vertices of $P_{6}$. Define a homomorphism $\psi:P_6 \to Y$ as follows:
        \begin{multicols}{4}
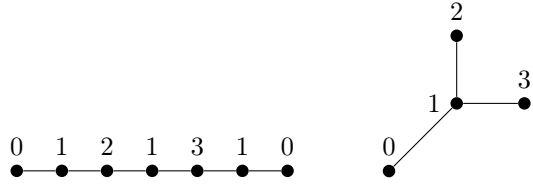
    \begin{itemize}[label={}]
        \item $a\mapsto 0$
        \item $b \mapsto 1$
        \item $c \mapsto 2$
        \item $d \mapsto 1$
        \item $e \mapsto 3$
        \item $f \mapsto 1$
        \item $g \mapsto 0$
    \end{itemize}
    \end{multicols}
    \begin{figure}[h!]
        \begin{tikzpicture}[scale=0.30]

  \node[fill,circle,scale=0.5,label={0}] (A) at (0, 3) {};
  \node[fill,circle,scale=0.5,label={1}]  (B) at (2, 3) {};
  \node[fill,circle,scale=0.5,label={2}]  (C) at (4, 3) {};
  \node[fill,circle,scale=0.5,label={1}]  (D) at (6, 3) {};
  \node[fill,circle,scale=0.5,label={3}]  (E) at (8, 3) {};
  \node[fill,circle,scale=0.5,label={1}]  (F) at (10, 3) {};
  \node[fill,circle,scale=0.5,label={0}]  (G) at (12, 3) {};

  \draw(A) -- (B);
  \draw (B) -- (C);
  \draw (C) -- (D);
  \draw (E) -- (D);
  \draw (E) -- (F);
  \draw (F) -- (G);

\end{tikzpicture}
        \qquad
        \begin{tikzpicture}[scale=0.30]

  \node[fill,circle,scale=0.5,label={0}] (A) at (0, 0) {};
  \node[fill,circle,scale=0.5,label={180:1}]  (B) at (3, 3) {};
  \node[fill,circle,scale=0.5,label={2}]  (C) at (3, 6) {};
  \node[fill,circle,scale=0.5,label={3}]  (D) at (6, 3) {};

  \draw(A) -- (B);
  \draw (B) -- (C);
  \draw (B) -- (D);

\end{tikzpicture}
    \end{figure}
    We consider $(P_{6},\psi)$ as a gadget with respect to $a$ and $g$ and argue that for every directed graph $D$ without isolated points the only homomorphisms $(P_{6},\psi)\to (D\star P_{6},\psi_D)$ in $\Gra/Y$ are the maps $\phi_{(u,v)}^{D,P_{12}}$ for $(u,v) \in E(D)$, obtaining universality by \Cref{thm:arrow}.

    So, let $\chi:(P_6,\psi)\to (D \star P_6,\psi_D)$ be a morphism. We immediately obtain that $\chi(c)=(u,v,c)$ for some $(u,v)\in E(D)$ and $\chi(e)=(u',v',e)$ for some $(u',v') \in E(D)$ since $c$ and $e$ are the only elements of $P_6$ whose image under $\psi$ is $2$ and $3$ respectively. Since the distance of $c$ and $e$ in $P_6$ is $2$, this implies that the distance of $(u,v,c)$ and $(u',v',e)$ is $\leq 2$ in $D \star P_6$; this forces that $(u,v)=(u',v')$. As a result, $\chi(a)=u, \chi(g)=v$, and $\chi(x)=(u,v,x)$ for $x \in \{b,c,d,e,f\}$. Consequently, $\chi$ is in fact equal to $\phi_{(u,v)}^{D,P_6}$.
\end{proof}

\begin{theorem}
    Let $G$ be a graph that contains one of $C_3,C_4,P_4,Y$ as a subgraph. Then $\Gra/G$ is algebraically universal.  
\end{theorem}

\begin{proof}
    Notice that if $G_1$ is a subgraph of $G_2$ then there is a natural fully faithful functor $F:\Gra/G_1 \to \Gra/G_2$. Therefore it suffices to prove this when $G$ is one of $C_3,C_4,P_4,Y$. This is precisely the content of the four lemmas above. 
\end{proof}

As a final remark, we note that in the case of general digraphs a straightforward generalisation of our results shows that $\Rel(2)/D$ is algebraically universal if, and only if, $D$ contains a loop, a directed $2$-cycle, or any orientation of the four graphs above. The case of arbitrary relational structures is less clear. 

\printbibliography

@book{pultr,
  title={Combinatorial, algebraic and topological representations of groups, semigroups and categories},
  author={Pultr, Ale{\v{s}} and Trnkov{\'a}, V{\v{e}}ra},
  volume={124},
  year={1980},
  publisher={Prague}
}

@article{nesetrilossona,
title = {Towards a characterization of universal categories},
journal = {Journal of Pure and Applied Algebra},
volume = {221},
number = {8},
pages = {1899-1905},
year = {2017},
issn = {0022-4049},
doi = {https://doi.org/10.1016/j.jpaa.2016.09.006},
url = {https://www.sciencedirect.com/science/article/pii/S0022404916301529},
author = {Jaroslav Nešetřil and Patrice {Ossona de Mendez}}
}

@article{birkhoff,
  title={On groups of automorphisms (Spanish)},
  author={Birkhoff, Garrett},
  journal={Rev. Un. Math. Argentina},
  volume={11},
  pages={155--157},
  year={1946}
}

@book{isbell,
author = {Isbell, John Rolfe},
language = {eng},
location = {Warszawa},
publisher = {Instytut Matematyczny Polskiej Akademi Nauk},
title = {Subobjects, adequacy, completeness and categories of algebras},
url = {http://eudml.org/doc/268641},
year = {1964},
}

@article{groot,
  title={Groups represented by homeomorphism groups I},
  author={de Groot, Johannes},
  journal={Mathematische Annalen},
  volume={138},
  number={1},
  pages={80--102},
  year={1959},
  publisher={Springer}
}

@article{sabidussi,
  title={Graphs with Given Infinite Groups.},
  author={Sabidussi, Gert},
  journal={Monatshefte f{\"u}r Mathematik},
  volume={64},
  pages={64--67},
  year={1960}
}

@article{pultr1964concerning,
  title={Concerning universal categories},
  author={Pultr, Aleš},
  journal={Commentationes Mathematicae Universitatis Carolinae},
  volume={5},
  number={4},
  pages={227--239},
  year={1964},
  publisher={Charles University in Prague, Faculty of Mathematics and Physics}
}

@article{hedrlin1966full,
  title={On full embeddings of categories of algebras},
  author={Hedrlín, Zdeněk and Pultr, Aleš},
  journal={Illinois Journal of Mathematics},
  volume={10},
  number={3},
  pages={392--406},
  year={1966},
  publisher={Duke University Press}
}

@book{neshell,
    author = {Hell, Pavol and Nešetřil, Jaroslav},
    title = "{Graphs and Homomorphisms}",
    publisher = {Oxford University Press},
    year = {2004},
    month = {07},
    isbn = {9780198528173},
    doi = {10.1093/acprof:oso/9780198528173.001.0001},
}

@article{dirgra,
author = {Hell, Pavol and Nešetřil, Jaroslav},
journal = {Archivum Mathematicum},
language = {eng},
number = {1-2},
pages = {47-54},
publisher = {Department of Mathematics, Faculty of Science of Masaryk University, Brno},
title = {Universality of directed graphs of a given height},
url = {http://eudml.org/doc/18257},
volume = {025},
year = {1989},
}

@misc{algunirel,
      title={Algebraically universal categories of relational structures}, 
      author={Ioannis Eleftheriadis},
      year={2023},
      eprint={2303.13274},
      archivePrefix={arXiv},
      primaryClass={math.LO}
}

@Article{setendofunctors,
author={Barto, Libor},
title={Finitary set endofunctors are alg-universal},
journal={Algebra Universalis},
year={2007},
month={08},
day={01},
volume={57},
number={1},
pages={15-26},
issn={1420-8911},
doi={10.1007/s00012-007-2011-7},
url={https://doi.org/10.1007/s00012-007-2011-7}
}

@Article{kucera,
author={Kučera, Luděk},
title={On universal concrete categories},
journal={Algebra Universalis},
year={1975},
month={12},
day={01},
volume={5},
number={1},
pages={149-151},
issn={1420-8911},
doi={10.1007/BF02485248},
url={https://doi.org/10.1007/BF02485248}
}

@article{hedrlinpultr,
author = {Hedrlín, Zdeněk and Pultr, Aleš},
journal = {Commentationes Mathematicae Universitatis Carolinae},
language = {eng},
number = {3},
pages = {377-400},
publisher = {Charles University in Prague, Faculty of Mathematics and Physics},
title = {On categorial embeddings of topological structures into algebraic},
url = {http://eudml.org/doc/16178},
volume = {007},
year = {1966},
}

@Article{mendelsohn,
author={Mendelsohn, Eric},
title={On a technique for representing semigroups as endomorphism semigroups of graphs with given properties},
journal={Semigroup Forum},
year={1972},
month={12},
day={01},
volume={4},
number={1},
pages={283-294},
issn={1432-2137},
doi={10.1007/BF02570800},
url={https://doi.org/10.1007/BF02570800}
}

\end{document}